\documentclass[12pt]{amsart}

\usepackage{latexsym, amssymb, amscd, amsfonts, amsrefs}
\usepackage[all,cmtip]{xy}
\usepackage{algorithm}
\usepackage[noend]{algpseudocode}
\usepackage{tikz-cd}

\usepackage{bm}
\usepackage[colorlinks=true,urlcolor=black,citecolor=black,linkcolor=black,%
pdftitle={Universal Scondaries for Rings of Invariants},%
pdfauthor={H.E.A. Campbell and David L. Wehlau},%
pdfsubject={Invariant theory},%
pdfkeywords={Secondary invariants, Reflection groups, Cohen-Macaulay rings}]{hyperref}
\usepackage{ifthen}

\hypersetup{bookmarksopen,linkcolor=blue,citecolor=magenta,colorlinks}

\DeclareMathAlphabet{\mathpzc}{OT1}{pzc}{m}{it}
\DeclareMathOperator{\res}{res}
\DeclareMathOperator{\cov}{covol}
\DeclareMathOperator{\rank}{rank}

\newboolean{bourbaki}
\setboolean{bourbaki}{false}

\newcounter{eqcounter}[section]
\renewcommand{\theeqcounter}{\arabic{section}.\arabic{eqcounter}}

\providecommand{\binom}[2]{{#1\choose#2}}

\renewcommand{\geq}{\geqslant}
\renewcommand{\leq}{\leqslant}

\def\GL{\text{GL}}

\ifthenelse{\boolean{bourbaki}}
{
\newcommand{\CC}{\mathbf{C}} 
\newcommand{\NN}{\mathbf{N}} 
\newcommand{\QQ}{\mathbf{Q}} 
\newcommand{\RR}{\mathbf{R}} 
\newcommand{\ZZ}{\mathbf{Z}} 
\newcommand{\FF}{\mathbf{F}} 
}
{
\newcommand{\CC}{\mathbb{C}} 
\newcommand{\NN}{\mathbb{N}} 
\newcommand{\QQ}{\mathbb{Q}} 
\newcommand{\RR}{\mathbb{R}} 
\newcommand{\ZZ}{\mathbb{Z}} 
\newcommand{\FF}{\mathbb{F}} 
}
\newcommand{\Sn}{{S_n}}
\newcommand{\KK}{K}
\newcommand{\OK}{{{\mathcal O}_K}}

\newtheorem{thm}{Theorem}[section]
\newtheorem{theorem}[thm]{Theorem}
\newtheorem{prop}[thm]{Proposition}

\newtheorem{cor}[thm]{Corollary}

\newtheorem{lem}[thm]{Lemma}

\newtheorem{example}[thm]{Example}

\theoremstyle{definition}
\newtheorem{definition}[thm]{Definition}
\newtheorem{defn}[thm]{Definition}

\newtheorem{remark}[thm]{Remark}

\usepackage{color}

\newcommand{\hVmodG}{{\widehat{V}/\!\!/G}}
\newcommand{\hVmodSigma}{{\widehat{V}/\!\!/\Sigma}}

\newcommand{\disc}{{\mathbf{disc}}}
\newcommand{\id}{{\mathrm{id}}}

\newcommand{\calb}{\mathcal{B}}

\newcommand{\calh}{\mathcal{H}}

\newcommand{\calm}{\mathcal{M}}

\newcommand{\calp}{\mathcal{P}}

\newcommand{\calr}{\mathcal{R}}
\newcommand{\cals}{\mathcal{S}}

\newcommand{\calz}{\mathcal{Z}}
\newcommand{\FP}{\operatorname{\#FP}}

\newcommand{\set}[1]{\{#1\}}
\newcommand{\eval}[1]{\mathbin{\kern -1.1pt \mid_{#1}}}
\newcommand{\semidirect}{\mathbin{\times \kern -1.6pt \vrule height 6pt width 0.5 pt depth 0pt}\ }

\newcommand{\no}{\operatorname{\#cycles}}
\newcommand{\Qi}{{\KK}}

\newcommand{\david}[1]{\textcolor{blue}{David: #1}}
\newcommand{\btheta}{{\overline{\theta}}}
\newcommand{\bth}{{{\theta_*}}} 

\newcommand{\QL}{{\QQ L}}
\newcommand{\RL}{{\RR L}}

\begin{document}
\pagestyle{plain} 
\title{When are Permutation Invariants Cohen-Macaulay?}

\begin{abstract}
  Over a field of characteristic 0, every ring of invariants of a finite group
is Cohen-Macaulay.  This is not true for fields of positive characteristic.  
We consider permutation representations and their invariant rings over 
fields $\mathbb{F}_p$ of prime order.   We give an efficient algorithm which for any given permutation representation, determines those primes $p$ for which the invariant ring over $\mathbb{F}_p$ is Cohen-Macaulay, using linear algebra over $\ZZ$.  A generalization of the classical discriminant associated to the alternating group is defined for subgroups of certain finite unitary complex reflection groups.
\end{abstract}



\dedicatory{}

\author{H.E.A. Campbell $^{1}$ and David L. Wehlau $^{1,2}$}
\address{$^1$Department of Mathematics and Statistics, Queen's University, Kingston ON, Canada}
\address{$^2$Department of Mathematics and Computer Science, Royal Military College of Canada, Kingston ON, Canada}
\email{eddy@unb.ca, wehlau@rmc.ca}

\maketitle
\tableofcontents
 
\section{Introduction} \label{sec: intro}
  Let $R$ be a ring and suppose the finite group $G$ acts on the graded
  polynomial ring $R[x_1,x_2,\dots,x_n]$.  The set of polynomials fixed pointwise by each element of $G$ forms a subring known as the ring of 
  invariants and denoted $R[x_1,x_2,\dots,x_n]^G$.  Here we consider whether
  this ring of invariants is Cohen-Macaulay.  We are mainly concerned with the case where $R$ is a finite field of prime order and $G$ acts by permuting 
  the variables $x_i$.

  If $R=k$ is a field with algebraic closure $\widehat{k}$ then we may view   $k[x_1,x_2,\dots,x_n]$ as a subring of  $\widehat{k}[x_1,x_2,\dots,x_n]$ which is the
  coordinate ring of the $n$ dimensional affine space 
  $\widehat{V}=\mathbb{A}^n=\widehat{k}^n$. 
    Here $\{x_1,x_2,\dots,x_n\}$ is a basis for the dual space $\widehat{V}^*$ and thus
    the variables $x_i$ are linear functions on $\widehat{V}$. 
       Let $\{e_1,e_2,\dots,e_n\}$ be the basis of $\widehat{V}$ dual to  $\{x_1,x_2,\dots,x_n\}$
   and write $V$ to denote the $k$-vector space spanned by $\{e_1,e_2,\dots,e_n\}$.
  Following the usual convention we write 
  $\widehat{k}[\widehat{V}]=\widehat{k}[x_1,x_2,\dots,x_n]$.
  More generally, we will write $R[V]$ to denote $R[x_1,x_2,\dots,x_n]$.
  The action of $G$ on $k[V]$ yields a corresponding  $G$-action
  on $V$ given by $(g\cdot f)(v) = f(g^{-1} v)$ for $g \in G$, $f \in k[V]$ and $v \in V$.  With this identification, the invariant polynomials are precisely the elements of $k[x_1,x_2,\dots,x_n]$
 which are constant on the $G$-orbits in $\widehat{V}$.

 By Hochster and Eagon (\cite{H-E}), $k[V]^G$ is Cohen-Macaulay if $k$ is any  field of characteristic 0,
 or if the characteristic of $k$ does not divide $|G|$.
 However if $p$ does divide $|G|$,  $k[V]^G$ may fail to be Cohen-Macaulay.
 Smith \cite{S} showed in 1996 that $k[V]^G$ is always Cohen-Macaulay if $k$ is a field and the dimension of $V$ is 3 or less.
 
  The case where $G$ acts by permuting the variables $x_1,x_2,\dots,x_n$ is of particular interest.  In this case $G$ is a subgroup of the full group of 
  permutations on $n$ letters, $\Sn$.   For $G \leq \Sn$, 
Blum-Smith and Marques\cite{BenandSophie} showed that $\ZZ[V]^G$ is Cohen-Macaulay if and only if the action of
$G$ on $V$ is generated by reflections and bi-reflections.
Reflections are elements whose fixed point set is a codimension 1 subspace of $V$ and bi-reflections are elements whose fixed point set is a codimension 2 subspace of $V$. 
The reflections in $S_n$ are the transpositions and the bi-reflections in $\Sn$ 
are the products of two different transpositions, which includes all 3-cycles.
Note that the identity element is neither a reflection nor a bi-reflection .

The bulk of this paper assumes that $G \leq \Sigma \leq \Sn$ where $\Sigma$ is a Young subgroup of $\Sn$,
i.e., $\Sigma=S_{k_1} \times S_{k_2} \times \dots \times S_{k_r}$ with $k_1+k_2+\dots+k_r=n$.  Often we will use
$\Sigma=\Sn$.
 It is well-known that $\ZZ[V]^{\Sigma}$ and $\QQ[V]^{\Sigma} = \ZZ[V]^{\Sigma}\otimes_{\ZZ} \QQ$ are both
 polynomial rings.
In fact $R[x_1,x_2,\dots,x_n]^{\Sigma}$ is a polynomial 
ring for all commutative rings $R$.  

A set of $R$-algebraically independent homogeneous generators for $R[V]^{\Sigma}$ is known as a set of \emph{primary invariants} for $R^G$.
In fact, any homogeneous system of parameters for $R[V]^G$ forms a set of primary invariants, but for this article we will only use algebra generators for $R[V]^{\Sigma}$ as primary invariants.

We can give a definition of Cohen-Macaulayness tailored for this setting. The ring of $G$-invariants $R[V]^G$ forms a finitely generated module over $R[V]^{\Sigma}$ and we have $R[V]^G = \sum_{j=1}^m R[V]^{\Sigma}\theta_j$  for some homogeneous $\theta_j \in R[V]^G$.  
A minimal set of homogenous module generators $\theta_1,\theta_2,\dots,\theta_m$ is called a set of {\em secondary invariants} for $R[V]^G$.  
It is known that $m \geq \ell :=[\Sigma:G]$, (see Derksen and Kemper \cite[Thm 3.9.1, pg 111]{D-K}).
The ring $R[V]^G$ is said to be \emph{Cohen-Macaulay} if it is free as a $R[V]^{\Sigma}$-module.
It can be shown that $R[V]^G$ is Cohen-Macaulay if and only if 
some (every) set of secondary invariants has cardinality $\ell$.
Thus $R[V]^G$ is Cohen-Macaulay if and only if 
$R[V]^G = \oplus_{j=1}^\ell R[V]^{\Sigma}\theta_j$ for some (every)
set of secondary invariants $\theta_1,\theta_2,\dots,\theta_\ell$.  

The above definition of Cohen-Macaulayness is sufficient for our purposes.  For a definition valid in more general settings we refer the reader
to Section~3 of Stanley's paper \cite{bible} or Section~2.1 of the paper of Blum-Smith and Marques \cite{BenandSophie}. 

We denote by $I(R)$ the ideal of $R[V]^G$ generated by the set of homogeneous elements of $R[V]^\Sigma$ of positive degree, i.e.,
 $I(R)=R[V]_+^\Sigma R[V]^G$ where $R[V]^\Sigma_+ = \oplus_{d=1}^\infty R[V]^\Sigma_d$.
 and write $\epsilon: R[V]^G \longrightarrow R[V]^G/I(R)$ to denote the canonical surjection. 
 Since $I(R)$ is homogeneous this yields surjections in each degree
$\epsilon_d: R[V]_d^G \longrightarrow R[V]_d^G/I(R)_d = (R[V]^G/I(R))_d$.
If $R$ is a field and $R[V]^G$ is Cohen-Macaulay, then any set of $\ell$ homogeneous invariants $\theta_1,\theta_2,\dots,\theta_\ell$
is a set of secondaries if and only if
$\{\epsilon(\theta_1),\epsilon(\theta_2),\dots,
\epsilon(\theta_\ell)\}$ is an $R$-vector space basis
for $R[V]^G/I(R)$, \cite[prop. 3.1]{bible} or \cite[lemma 3.7.1]{D-K}

By Hochster and Eagon \cite[Proposition 19]{H-E},
$\QQ[V]^G$ is Cohen-Macaulay, i.e.,  it is free (of rank $\ell$) as a $\QQ[V]^{\Sigma}$-module.
If $p \in \ZZ$ is a prime then, reducing mod $p$, the representation of $G$ descends naturally to a representation of $G$
defined over $\FF_p=\ZZ/(p)$.  For primes $p$ not dividing $|G|$ the ring $\FF_p[V]^G=\FF_p[x_1,x_2,\dots,x_n]^G$ is always Cohen-Macaulay, again by Hochster and Eagon\cite[Proposition 19]{H-E}.
However for those primes $p$ which do divide $|G|$, the ring $\FF_p[V]^G$ may or may not be Cohen-Macaulay.

Our initial interest in this was inspired by \cite{BT}, in particular, the idea of evaluating a set of secondaries at an integer point $z$ with trivial stabilizer, which, in a certain sense, translates the discussion of the question of the Cohen-Macaulay property for one of these invariant rings into linear algebra over $\ZZ$ and lattice algorithms.  Notably, we are able to avoid the use of Gr\"obner or Khovanskii bases.

We have organized this article as follows.
  In Section~\ref{summary} we give a list summarizing our main results.
Section~\ref{preliminaries} reviews basic facts about the invariants of permutation representations.   
Following Blum-Smith and Marques~\cite{BenandSophie}, 
in Section~\ref{good and bad}, we
divide the set of primes into good and bad primes according to whether 
or not $\FF_p[V]^G$ is Cohen-Macaulay.  
We also introduce the main tool in our work, the square matrix $M(\bth)$
associated to a set of secondary invariants $\bth$.   
In Section~\ref{existence} we 
show that there exists a universal set of secondaries, i.e., a set of $\ell$ invariants in $\ZZ[V]^G$ which
reduce modulo $p$ to a set of secondary invariants for all good primes.

In Section~\ref{discriminant} we generalize the definition of the {\em discriminant} associated
to the alternating subgroups of 
the symmetric group to subgroups, $G$, of a finite complex unitary group, $\Sigma$, in the case that the ring of integers of the associated number field is a principal ideal domain.  The construction yields a semi-invariant $\Delta(G)$ built from the linear forms defining the reflecting hyperplanes of $\Sigma$ and an analysis of the action of $\Sigma$ on the cosets of $G$ in $\Sigma$.  We believe these results to be of independent interest.  We show that this semi-invariant is the primitive part of $\det M(\bth)$ in Section~\ref{section:divides}.  In Section~\ref{section:constructing}  we present an algorithm for constructing a universal set of secondaries and Section~\ref{conclusion} provides some concluding remarks.

\section{Summary of the Main Results}\label{summary}
 Suppose that $V$ is the canonical permutation representation of a group $G \leq \Sigma$ for a Young subgroup $\Sigma \leq \Sn$ and let
$\bth=(\theta_1,\theta_2,\dots,\theta_\ell)$ be a sequence of elements of $\ZZ[V]^G$ where $\ell$ denotes the index of $G$ in $\Sigma$.  
Let $\btheta_j$ denote the reduction modulo $p$ of the integer polynomial $\theta_j$.
We construct an $\ell\times\ell$ matrix $M(\bth)$ with entries in $\ZZ[V]^G$.  

Our main results may be summarized as follows.
\begin{enumerate}
   \item  
   $\FF_p[V]^G = \oplus_{j=1}^\ell \FF_p[V]^{\Sigma} \btheta_j$
   if and only if $p$ does not divide the content, $\mho(\bth)$, of 
   the polynomial $\det M(\bth)$.
 \item There exists a set of invariants $\bth \in \ZZ[V]^G$ such that
 $\FF_p[V]^G$ is Cohen-Macaulay if and only if $p$ does not divide 
 the content, $\mho(\bth)$, of the polynomial $\det M(\bth)$.
 Furthermore $\FF_p[V]^G = \oplus_{j=1}^\ell \FF_p[V]^{\Sigma} \btheta_j$
 for all primes $p$ not dividing $\mho(\bth)$.
  \item We give an algorithm for constructing such a set of universal secondary invariants.
 \item We give an algorithm for finding the content of 
  $\det M(\bth)$ working over the ring of integers.  By working with evaluation at an appropriately chosen point, we are able to avoid 
  constructing $M(\bth)$.
  \item For a subgroup, $G$, of a finite complex unitary reflection group, $\Sigma$,
  we define a natural semi-invariant, the $G$-discriminant $\Delta(G)$, depending on the generating reflections 
  of $G$ and the action of $\Sigma$ on $\Sigma/G$.
  \item  For a subgroup $G$ of such a reflection group $\Sigma$ and
  for $\theta_1,\theta_2,\dots,\theta_\ell$ a set of secondaries
  for the $G$-invariants as a module over the $\Sigma$-invariants,
  we show that $\det M(\bth)$ is a scalar multiple of $\Delta(G)$. 
\end{enumerate}

\section{ Preliminaries.} \label{preliminaries}

For an introduction to the invariant theory of finite groups and descriptions of the important results we refer the reader to the wonderful survey article by Richard Stanley \cite{bible}.  Both computational and theoretical aspects of invariant theory are surveyed in the book of Derksen and Kemper \cite{D-K}.  Facts about invariants of finite groups specifically in positive characteristic may be found in our book \cite{C-W}.

Since the group action preserves degree, rings of invariants are graded rings:
$R[V]^G = \sum_{d=0}^\infty R[V]^G_d$.  Over a field $k$ we have   
the Hilbert series of $k[V]^G$.  This is the power series in the indeterminate $\lambda$ 
$$\calh(k[V]^G,\lambda) =
\sum_{d=0}^\infty \dim_{k} (k[V]^G_d)\, \lambda^d\ .$$  This is always a 
rational function and for a subgroup $G \subseteq \Sigma$ it may be written in the  
form 
$$\calh(k[V]^G,\lambda)
=\frac{\sum_{j=1}^m \lambda^{a_j}}{\prod_{i=1}^n(1-\lambda^{d_i})}\ .$$  Here $m \geq \ell$ where $\ell=[\Sn:G]$, $a_j \in \NN$, and the polynomial ring $\QQ[V]^\Sigma$ is generated by homogeneous elements of degrees $d_1,d_2,\dots,d_n$.  Furthermore we have $\deg \theta_j = a_j$ for $j=1,2,\dots,m$ for any set of secondaries $\theta_1,\theta_2,\dots,\theta_m$. Molien's Theorem (see Stanley, \cite[Theorem 2.1]{bible}) provides a simple formula for the Hilbert series of a ring of invariants when $|G|$ is invertible in $k$. Thus Molien's Theorem may be used to find  the degrees of secondary invariants for $\QQ[V]^G$.

For permutation representations defined over $\QQ$, Kronecker\cite{K} proved 
the existence of a set of secondaries of degree at most $\binom{n}{2}$.
B.~Schmid[Proposition 9.2]\cite{Schmid} showed this upper 
bound is valid over any field of characteristic zero.  
Her statement of this result, and its proof, are slightly misstated but 
her method gives a non-constructive proof that secondaries for a subgroup of $\Sn$ in this setting
have degree at most $\binom{n}{2}$.

A few years later
M.~G\"obel proved a theorem \cite[Theorem 3.11]{Gobel} which is usually described as asserting that,
in all characteristics, the ring of invariants of a 
permutation representation
is always generated as an algebra by homogeneous invariants of degree at most $\max\{n,\binom{n}{2}\}$.  In fact, his proof provides an algorithm which 
constructs homogeneous secondary invariants of degree at most $\binom{n}{2}$ for $R[V]^G$ when 
$G \leq \Sn$.  This algorithm works whether or not the invariant ring is Cohen-Macaulay. 

For permutation representations there is a natural vector space basis for 
$(\QQ[V]^G)_d$ given by orbit sums.  Let $m=x_1^{i_1} x_2^{i_2} \cdots x_n^{i_n}$ be a monomial of degree $d$.  
The $G$-orbit of $m$ is the finite set of monomials 
$G\cdot m = \{g \cdot m \mid g \in G\}$.
The sum of these monomials is a monic invariant 
which we denote by $\calz(m)$ and call the orbit sum of $m$:
$$\calz(m) := \sum_{m' \in G\cdot m} m'\ .$$

The set of orbit sums of monomials of degree $d$ is clearly a basis for $(\QQ[V]^G)_d$.  This shows that 
$\dim_{\QQ} (\QQ[V]^G)_d$ is the number of orbits of the $G$-action on the set of monomials of degree $d$.

Furthermore if we reduce mod $p$ the reduced orbit sums $\overline{\calz(m)}$ form a basis for $(\FF_p[V]^G)_d$.  Thus
$\calh(\QQ[V]^G,\lambda) = \calh(\FF_p[V]^G,\lambda)$.
In particular the degrees of secondaries for 
$\QQ[V]^G$ and for $\FF_p[V]^G$ are equal if $\FF_p[V]^G$ is Cohen-Macaulay.  

\section{Good and Bad Primes}\label{good and bad}

We consider a group $G$ acting by permuting a set of $n$ indeterminants
$x_1,x_2,\dots,x_n$.  We divide the set of all integer primes $p$ 
into two disjoint classes by whether or not  
$\FF_p[x_1,x_2,\dots,x_n]^G$ is Cohen-Macaulay.

\begin{defn}
  Let $p \in \ZZ$ be a prime. 
  and let $\FF_p$ denote the field $\FF_p = \ZZ/(p)$ of order 
  and characteristic $p$.
  Consider a set of secondary invariants 
  $\bth = (\theta_1,\theta_2,\dots,\theta_\ell)$
  for $\QQ[V]^G$ with each $\theta_i \in \ZZ[V]^G$.  
  Thus $\QQ[V]^G = \oplus_{i=1}^\ell \QQ[V]^{\Sigma} \theta_i$.
  Clearing denominators and then reducing $\theta_i$ modulo $p$ yields $\btheta_i \in \FF_p[V]^G$.  
  Following Blum-Smith and Marques~\cite{BenandSophie},
  we say $p$ is \emph{good for $\bth$}  
  if $\FF_p[V]^G=\oplus_{i=1}^\ell \FF[V]^{\Sigma} \btheta_i$.  Otherwise $p$ is \emph{bad} for $\bth$. 
  Note that if $p$ is good for $\bth$ then $\FF_p[V]^G$ is Cohen-Macaulay but
  the converse is false in general: if $p$ is bad for $\bth$
  it may still be the case that $\FF_p[V]^G$ is Cohen-Macaulay (see
  Example~\ref{ex:alternating group}).
  
  More generally a prime $p$ is \emph{bad} if it is bad for all choices 
  of $\bth$, i.e., if $\FF_p[V]^G$ is not Cohen-Macaulay and $p$ is 
  \emph{good} if it is good for some choice of $\bth$, i.e., if
  $\FF_p[V]^G$ is Cohen-Macaulay.
\end{defn}

We will use $\gamma_1, \gamma_2,\dots,\gamma_\ell$ to denote a set of left 
coset representatives of $G$ in $\Sigma$.

Given a set of secondaries $\bth=(\theta_1,\theta_2,\dots,\theta_\ell)$ for 
$\QQ[V]^G$ with each $\theta_i \in \ZZ[V]^G$ we write $M(\bth)$ to 
denote the $\ell \times \ell$ matrix 
$M(\bth) = [\gamma_i \theta_j]_{\substack{1 \leq i \leq \ell\\1 \leq j \leq \ell}}$.
Fix a prime $p$.  We write $M(\overline\bth)$ to denote
the matrix obtained from $M(\bth)$ by reducing each entry modulo $p$, 
i.e., $M(\overline{\theta}_*) 
= [\gamma_i \btheta_j]_{\substack{1\leq i \leq \ell \\ 1 \leq j \leq \ell}}$.

For an integer polynomial $f\in \ZZ[V]$ the \emph{content} of $f$
is the $\gcd$ of the coefficients of $f$.  The \emph{primitive part} of such a polynomial is the quotient of the polynomial by its content. 
\begin{defn}\label{content}
We denote the content of the polynomial $\det M(\bth)$ by 
$\mho(\bth)$. 
We call the integer $\mho(\bth)$ the \emph{deficiency} of the 
sequence $\theta_1,\theta_2,\dots,\theta_\ell$.
\end{defn}

\subsection{Bad Primes}
Fix a set of secondaries $\bth=(\theta_1,\theta_2,\dots,\theta_\ell)$ for
 $\QQ[V]^G$ with each $\theta_i \in \ZZ[V]^G$.
 In this section we will show that a prime $p$ is bad for 
 $\bth$ if and only if $p$ divides $\mho(\bth)$.

Recall that $a_j = \deg(\theta_j)$.
For each degree $d$ we have the map 
$$\rho_d : \oplus_{j=1}^\ell \ZZ[V]^{\Sigma}_{d-a_j}  
  \to \ZZ[V]_d^G$$ given by 
  $\rho_d(f_1,f_2,\dots,f_\ell) = f_1\theta_1 + f_2\theta_2+ \dots +f_\ell\theta_\ell$.

Reducing $\rho_d$ modulo $p$ gives the map 
$$\overline{\rho}_d : \oplus_{j=1}^\ell \FF_p[V]^{\Sigma}_{d-a_j}  
  \to \FF_p[V]_d^G$$ given by 
  $\overline{\rho}_d(f_1,f_2,\dots,f_\ell) = f_1\btheta_1 + f_2\btheta_2+ \dots +f_\ell\btheta_\ell$.

 Now we claim $p$ is good for $\bth$ if and only if the $\ell$ elements of $\overline{\theta}_*$  form a module basis for 
  $\FF_p[V]^G$ over $\FF_p[V]^{\Sigma}$ if and only if  $\overline{\rho}_d$ is surjective for all $d$.  We see
  this as follows.
   Since $\QQ[V]^G$ is Cohen-Macaulay, the linear transformation
 $$\rho_d \otimes \QQ : \oplus_{j=1}^\ell \QQ[V]^{\Sigma}_{d-a_j}  
  \to \QQ[V]_d^G$$ is an isomorphism for each $d$. 
From the equality of the Hilbert series we know that $\dim_{\FF_p} \left(\oplus_{j=1}^\ell \FF_p[V]^{\Sigma}_{d-a_j}\right) = \dim_{\FF_p} \FF_p[V]^{G}_{d}$.  Hence 
  $\overline{\rho}_d$ is surjective if and only if it is
  injective.

\begin{lem}\label{lem: tuple}
  The multi-homogeneous $\ell$-tuple $(f_1,f_2,\dots,f_\ell)\in (\FF_p[V]^{\Sigma})^\ell$ lies in the kernel of $\overline{\rho_d}$ if and only it lies in the kernel of the matrix $M(\overline{\theta}_*)$.  
 \end{lem}

\begin{proof}
  The sufficiency is obvious.\\ 
  Suppose then that $f_1\btheta_1 + f_2\btheta_2+ \dots +f_\ell\btheta_\ell=0$.
  Therefore 
$$0= \gamma_i \cdot \sum_{j=1}^\ell f_j \btheta_j 
    = \sum_{j=1}^\ell (\gamma_i\cdot f_j) (\gamma_i\cdot \btheta_j)
    = \sum_{j=1}^\ell (f_j) (\gamma_i\cdot \btheta_j)$$
    for all $i$.  Thus $(f_1,f_2,\dots,f_\ell)$ lies in the kernel of $M(\overline{\theta}_*)$.
\end{proof}

As a consequence of Lemma~\ref{lem: tuple} we have the following result.
\begin{thm}
The prime $p\text{ divides }\mho(\bth)$ if and only if $p$ is bad for $\bth$. 
\end{thm}
\begin{proof}
Clearly if $\overline{\rho_d}$ is not injective for some 
$d$ then $\det M(\overline{\theta}_*) = 0$.  Conversely suppose that
$\det M(\overline{\theta}_*) = 0$.  Since $M(\overline{\theta}_*)$ is square
matrix with entries from $\FF_p[V]$, it follows that 
$M(\overline{\theta}_*)$ has a non-zero kernel over the field
$\FF_p(V)$.  

This kernel is stable under the action of $\Sigma$.  To see this suppose 
$(f_1,f_2,\dots,f_\ell) \in \FF_p(V)^\ell$ lies in the kernel of
$M(\overline{\theta}_*)$,
i.e, $\sum_{i=1}^\ell \gamma_j \theta_i f_i = 0$ for all $j$.  Consider
$(\sigma f_1,\sigma f_2,\dots,\sigma f_\ell)$.  For each $j$ there exists $k$
and $g_j \in G$ such that 
$\sigma^{-1}\gamma_j = \gamma_k g_j$.
Thus 
$\sum_{i=1}^\ell (\gamma_j \theta_i) (\sigma f_i)= 
\sigma \sum_{i=1}^\ell (\sigma^{-1} \gamma_j\theta_i) f_i
=\sigma \sum_{i=1}^\ell \gamma_k g_j \theta_i f_i
=\sigma \sum_{i=1}^\ell \gamma_k \theta_i f_i
= \sigma \cdot 0 = 0$.
Hence the kernel of $\det M(\overline{\theta}_*)$ in $\FF(V)^\ell$ is $\Sigma$-stable.

By Milne\footnote{We thank the anonymous referee for suggesting this reference.} \cite[Lemma 3.4]{M} this implies that $M(\overline{\theta}_*)$ has a non-zero kernel over the field
$\FF_p(V)^{\Sigma}$.  
Thus it has a non-zero kernel over $\FF_p[V]^\Sigma$
and therefore $\overline{\rho_d}$ is not injective for some 
$d$.

Hence
  \begin{align*}
    p\text{ divides }\mho(\bth) &\iff \det M(\overline{\theta}_*) = 0\\
    & \iff \overline{\rho_d}\text{ is not injective for some }d \\
    &\iff \overline{\rho_d}\text{ is not surjective for some }d\\
    &\iff p\text{ is a bad prime for }\bth \ .
  \end{align*}
\end{proof}

 \subsection{Good Primes}\label{good primes}\ \\
 In this section we prove, from another point of view, that if $p$ is good for $\bth$, that is, if $\overline{\theta}_*$ is a set of secondaries for $\FF_p[V]^G$, then $p$ does not divide $\mho(\bth)$. 
    Suppose we have chosen secondary invariants 
    $\bth \in \ZZ[V]$
    for $\QQ[V]^G$ which reduce to secondary invariants for $\FF_p[V]^G$.

 Let $\widehat{\FF}$ denote the algebraic closure of $\FF_p$ and write $\widehat{V} := V \otimes \widehat{\FF}$.
 Dual to the inclusions of rings $$\widehat{\FF}[\widehat{V}]^{\Sigma} \hookrightarrow \widehat{\FF}[\widehat{V}]^G \hookrightarrow \widehat{\FF}[\widehat{V}]$$
  are the morphisms of affine varieties
$$\widehat{V} \stackrel{\pi_G}{\longrightarrow}\hVmodG \stackrel{\phi}{\longrightarrow}  \hVmodSigma\ .$$

Then $\pi_{\Sigma} = \phi \circ \pi_G$ is the map giving the 
quotient by $\Sigma$.
  
  Consider $z \in \widehat{V}$ such that the $\Sigma$-orbit $\{g\cdot z \mid g \in\Sigma\}$ has cardinality $|\Sigma|$.
  Such a point $z$ is a point which is not fixed by any non-trivial element 
  of $\Sigma$, i.e., it is any point not lying on any of the reflection hyperplanes
  associated to $\Sigma$.  
   Then $\phi^{-1}(\pi_{\Sigma}(z))$ is a set of $\ell$ points in $\hVmodG$.    Concretely, the 
   $\Sigma$-orbit of $z$ breaks up into $\ell$ different $G$-orbits, and the $\ell$ points of 
   $\phi^{-1}(\pi_{\Sigma}(z))$ are the points in $\hVmodG$ corresponding to these orbits.   
   Write 
   $X := \{w_1,w_2,\dots, w_{\ell}\}$ where $w_i=\gamma_i^{-1} z$. The $\ell$ points of $X$ 
   are in one-to-one correspondence with these $\ell$ orbits. 
   Then $M(\overline{\theta}_*)|_z = [(\gamma_i \overline{\theta}_j)(z)]_{1 \leq i,j \leq \ell} = [\overline{\theta}_j(w_i)]_{1 \leq i,j \leq \ell}$. 
   
  For each $i=1,2,\dots,\ell$ consider the vector $\overline{\theta}_i|_X=
(\overline{\theta}_i(w_1),\overline{\theta}_i(w_2),\dots,\overline{\theta}_i(w_\ell))$.
     Showing that
   $\det (\overline{M}(\bth)|_z) \ne 0$ is equivalent to showing that these $\ell$ vectors are linearly independent over 
   $\widehat{\FF}$.

  The set $X$ is a closed subset of $\hVmodG$.  This implies  the restriction map from 
  $\hVmodG$ to $X$
  corresponds to a surjection onto the ring of functions on $X$.  But since $X$ is a set of $\ell$ points, this ring of functions is the ring 
  $\widehat{\FF}^\ell$ (with componentwise multiplication).   We denote this restriction map by
    $$
        \res_X:\widehat{\FF}[\widehat{V}]^G \twoheadrightarrow \widehat{\FF}^\ell\, .
    $$
    Since $p$ is a good prime we have
  $\FF_p[V]^G = \oplus_{i=1}^\ell \FF_p[V]^{\Sigma}\btheta_i$ and thus 
   $\widehat{\FF}[\widehat{V}]^G = \oplus_{i=1}^\ell \widehat{\FF}[\widehat{V}]^{\Sigma}\btheta_i$.  Now for $f \in \widehat{\FF}[\widehat{V}]^{\Sigma}$, the restriction of $f$ to $X$ is just 
   $(f(z),f(z),\dots,f(z))$.  Hence the image of the restriction of $\widehat{\FF}[\widehat{V}]^{\Sigma}$ to $X$ is the same as the constant functions 
   on $X$.  This shows that the image of restricting $\widehat{\FF}[\widehat{V}]^G$ is spanned by the $\ell$ functions
   ${\btheta_1}\rvert{_X},{\btheta_2}\rvert_X,\dots,{\btheta_\ell}\rvert_X$.  Since we know the ring of functions on $X$ has 
   vector space dimension $\ell$ it follows that these $\ell$ restrictions must be linearly independent over $\widehat{\FF}$.
   Since these vectors are the columns of $M(\overline{\theta}_*)\rvert_z$, it follows that $\det M(\overline{\theta}_*)\rvert_z \ne 0$.  This shows that
   if $p$ is a good prime for $\bth$  
   then $p$ does not divide $\mho(\bth)$.

 \begin{remark}\label{rem: lin ind over Q}
   Since the set of vectors 
   $\{\overline{\theta}_1|_X,\overline{\theta}_2|_X,\dots,\overline{\theta}_\ell|_X\}$
   is linearly independent over $\widehat{\FF}$ for every good prime $p$, it follows that the set of integer vectors 
  $\{\theta_1|_X,\theta_2|_X,\dots,\theta_\ell|_X\}$ is linearly independent over $\QQ$.
 \end{remark}

\begin{example}\label{ex:alternating group}
  Consider the alternating group
$G=A_n \le \Sigma=\Sn$.  Since $A_n$ is generated by 3-cycles which are bi-reflections, the theorem of Blum-Smith and Marques implies that $\FF_p[V]^{A_n}$ is Cohen-Macaulay for all primes $p$.  This result was first proved by Vic Reiner \cite{ReinerThesis}.  Here we show how it follows from our methods.

There are two natural choices for secondary invariants: 
$\{1,\disc\}$ where 
    $$
        \disc = \prod_{1\leq i<j\leq n}(x_i-x_j)$$ 
and $\{1,\calz(m)\}$ where $m=\prod_{i=1}^{n-1} x_i^{n-i}$.  

Take $\sigma \in \Sn \setminus A_n$.
Then $\sigma\cdot \disc = -\disc$ and $\sigma\cdot \calz(m) = \calz({m'})$ where 
$m'= x_n\prod_{i=1}^{n-2} x_i^{n-i}$.

Thus $M(1,\disc)=
\begin{pmatrix}
    1 & \disc\\
    1 & -\disc
\end{pmatrix}$
and $\det M(\bth)=-2\,\disc$ and 
$\mho(1,\disc)=2$.
Thus 2 is bad for $(1,\disc)$ and all other primes are good.

Conversely $M(1,\calz(m))=
\begin{pmatrix}
    1 & \calz(m)\\
    1 & \calz({m'})
\end{pmatrix}$
and $\det M(\bth)=\calz({m'})-\calz({m}) =-\disc$ and 
$\mho(1,\calz(m))=1$.
Hence all primes are good for $(1,\calz(m))$.
\end{example}

\section{Existence of Universal Secondaries}\label{existence}

In the previous example, we saw that there was a choice of secondary invariants
which was good for all primes.
Here, we want to show that, in general, there is a set of secondary invariants
which is good for all good primes $p$, i.e., a set
$\theta_1,\theta_2,\dots,\theta_\ell \in \ZZ[V]^G$ for which
$\FF_p[V]^G = \oplus_{i=1}^\ell \FF_p[V]^{\Sn} \btheta_i$ for all good primes.
We call such a set of secondaries a set of \emph{universal secondaries}.
\begin{theorem}   
A set of universal secondaries $\bth$ always exists. 
\end{theorem}

\begin{proof}
Let $S$ be the (finite) set of bad primes and write $\ZZ_S$ to denote the integers localized at the multiplicative set generated by $S$.

Recall that $\calh(k[V]^G,\lambda)
=\frac{\sum_{j=1}^\ell \lambda^{a_j}}{\prod_{i=1}^n(1-\lambda^{d_i})}$ for $k=\FF_p$ and for $k=\QQ$.
Since the Hilbert series is the same for all these $k$  we see that the number of secondaries of degree $d$, is the same for all $k$ for which $k[V]^G$ is Cohen-Macaulay.  We denote this number by $\tau(d):=\#\{j\mid a_j = d\}$.

Consider the exact sequence
$$
\bigoplus_{i=1}^d(\ZZ_S[V]^{\Sigma}_{i} \otimes \ZZ_S[V]^G_{d-i}) \xrightarrow{~\nu_d~} \ZZ_S[V]^G_d 
\xrightarrow{~\mu_d~} Q_d \longrightarrow 0
$$
where $Q_d$ is the co-kernel of $\nu_d$. 

For every good prime $p$, tensoring with $\FF_p = \ZZ/(p)$ yields the exact sequence
$$
\bigoplus_{i=1}^d(\FF_p[V]^{\Sigma}_{i} \otimes \FF_p[V]^G_{d-i}) \longrightarrow \FF_p[V]^G_d \longrightarrow Q_d\otimes \FF_p \longrightarrow 0 \ .$$
The cokernel $Q_d\otimes \FF_p$ is an $\FF_p$ vector space of dimension $\tau(d)$ for all
good primes $p$. 
Since $\ZZ_S[V]^G_d$ is a finitely generated $\ZZ_S$-module,
$Q_d$ is also a finitely generated $\ZZ_S$-module.  Hence the classification of finitely generated modules over a PID applies.  If there were a prime $p$ of $\ZZ_S$ for which $Q_d$ had $p$-torsion 
then the rank of $Q_d \otimes \FF_p$ would exceed $\tau(d)$. 
Since no such prime $p$ exists, $Q_d$ is a torsion free $\ZZ_S$-module.
Thus $Q_d$ is a free $\ZZ_S$-module since it is a finitely generated torsion free module over 
the PID $\ZZ_S$.

Since $Q_{d}$ is a free $\ZZ_S$-module, the map $\mu_d$ splits.  
Choosing a $\ZZ_S$-basis $f_1,\dots,f_{\tau(d)}$ for $Q_d$ and a section 
$\psi_d:Q_d \to \ZZ_S[V]^G_d$ of $\mu_d$ we may use $\psi_d(f_i) \in \ZZ_S[V]^G$ for $i=1,2,\dots,\tau(d)$ as the secondaries of degree $d$.  Clearing denominators, we may assume that each of these secondaries lies in $\ZZ[V]^G$.  Using such $\psi_d(f_i)$ as the secondaries $\theta_j$ we get 
$\ZZ_S[V]^G = \oplus_{j=1}^\ell \ZZ_S[V]^{\Sigma} \theta_j$
 and for all good primes $p$ we have  $\FF_p[V]^G = \oplus_{j=1}^\ell \FF_p[V]^{\Sigma} \btheta_j$.
\end{proof}

Thus we have proved 
\begin{cor}
  For every set of universal secondaries, $\bth$, the ring
    $\FF_p[V]^G$ is Cohen-Macaulay if and only if $p$ does not divide $\mho(\bth)$.
\end{cor}
\begin{definition}
    In Section~\ref{section:constructing} we give an algorithm which constructs a set of secondary invariants, $\bth$, such that $\mho(\bth)$ divides $\mho(\eta_*)$ for every set of secondaries $\eta_*$.  We define $\mho(G)$ to be this unique (up to sign) division minimum integer
    dividing $\mho(\eta_*)$ for all sequences of secondaries $\eta_* \subset \ZZ[V]^G$.
\end{definition}
Ben Blum-Smith has observed that the existence of a set of universal secondary invariants 
  can also be derived as a consequence of Theorem~7 from the recent paper \cite{AA} of  Areej Almuhaimeed.

  In Sections~\ref{section:divides} and \ref{section:constructing} below, we will give an algorithm for finding a set of universal
 secondaries and an efficient algorithm for determining $\mho(G)$ from a universal set.

 \begin{example}\label{ex: group of order 20}
The symmetric group $S_5$ contains a unique (up to conjugacy) subgroup $G$ of order 20.   This group is generated by
$\{(1, 5, 4, 2, 3),~(2, 3, 4, 5)\}$ and does not have a set of generators consisting of reflections and bi-reflections.  Thus the theorem of Blum-Smith and Marques implies that there is at least one bad prime.  Since this prime divides $|G|$ we see that 2 and/or 5 are bad primes for this group.  Using the algorithm of Section~\ref{section:constructing} with $\ZZ[V]=\ZZ[x_1,x_2,x_3,x_4,x_5]$ we get the following set of universal secondaries
$\theta_1=1$,
$\theta_2=\calz(x_1^2 x_2 x_3)$, 
$\theta_3=\calz(x_1^2 x_2^2 x_3)$, 
$\theta_4=\calz(x_1^3 x_2^2 x_3)$, 
$\theta_5=\calz(x_1^3 x_2^2 x_3 x_5)$, and
$\theta_6=\calz(x_1^4 x_2^3 x_4)$. 
This yields $\mho(\bth)=2$ (and thus $\mho(G)=2$). 
Thus $\FF_p[V]^G$ is Cohen-Macaulay if and 
only if $p\neq 2$.

Since $\ZZ[V]^G$ is not Cohen-Macaulay, we know there is a $G$-invariant of degree less than or equal $\binom{5}{2}=10$ that is not in $\calm:=\oplus_{j=1}^\ell \ZZ[V]^{\Sn}\theta_j$.  Magma computation shows us that $f = \calz(x_1^4 x_2^3 x_3^2x_4)$ is not in $\calm$ but we have $2f \in \calm$.  
\end{example}

\section{The action of a reflection on the cosets of $G$}\label{reflections}

Since the results proved here may be independent interest, for the next three sections 
we work in the more general setting where the group $G$ is a subgroup of an arbitrary finite unitary complex reflection group $\Sigma$.
Thus $\Sigma$ is a finite subgroup generated by reflections.  As usual, a reflection is a non-identity group element fixing a hyperplane pointwise.  The group $\Sigma$ is defined over the ring of integers $\OK$ of some number field $\KK$, i.e., 
$\Sigma \leq \GL(n,\OK)$.  
The ring $K[V]^\Sigma$ is a polynomial ring and we
use a set of algebra generators for $K[V]^\Sigma$ 
which lie in $\OK[V]$ as primary invariants.  For 
the secondary invariants $\theta_*$ 
we use elements of $\OK[V]$ which form a set of minimal module generators
for $K[V]^G$ as a module over $K[V]^\Sigma$.

In this and the next two sections, we show that 
$\det{M}(\theta_*)$ has a geometric interpretation in terms of the 
reflecting hyperplanes of $\Sigma$ and the action of $\Sigma$ on the cosets 
of $G$ in $\Sigma$.  In particular, in this section we derive a formula for 
the number of reflections contained in $G$.  

Let $\calp$ denote the set of all reflecting hyperplanes for $\Sigma$.  
Associated to each hyperplane $P \in \calp$ is a linear form $L_P\in\KK[V]$ which cuts out $P$.  Clearing denominators we may assume that $L_P 
\in \OK[V]$.  If $\OK$ is a PID, we will choose the linear form
 $L_P$ to be primitive.
Let $C_P$ denote the cyclic subgroup consisting of those elements of $\Sigma$ which fix $P$ pointwise.  For each hyperplane $P$ fix an element $g_P$ generating $C_P$.

Note that $\Sigma$ permutes the elements of $\calp$ since if $\sigma \in \Sigma$
then $\sigma g_P \sigma^{-1}$ fixes the hyperplane $\sigma P$.   
Suppose that $\calb$ is a $\Sigma$-orbit in $\calp$ and fix $P_0 \in \calb$. 
Put $C = C_{P_0}$ and $g=g_{P_0}$.   Write $c = |C|$.

We write $\calr_\calb(\Sigma)$ to denote the set of all elements (reflections) in $\Sigma$ that are conjugate to an element of $C \setminus \{\id\}$.  
Thus $\calr_\calb(\Sigma) = \sqcup_{P \in \calb} (C_P \setminus \{\id\})$. 
Put $\calr_\calb(G) := G \cap \calr_\calb(\Sigma)$.

Recall that $\cals = \set{\gamma_1G, \gamma_2 G, \dots, \gamma_\ell G}$ denotes the set of
left cosets of $G$ in $\Sigma$.   Then $\Sigma$ acts on $\cals$ by left multiplication and permutes the cosets.

For $\sigma \in \Sigma$, write $b_i(\sigma)$ to denote the number of cycles of order $i$ 
in the cycle decomposition for the action of $\sigma$ on $\cals$, i.e., $b_i(\sigma)$ is the number of 
orbits in $\cals$ of size $i$ under the action of $\sigma$.

Let $\cals/C$ denote the set of $C$-orbits in $\cals$.  Then $|\cals/C| = \sum_{i=1}^\ell b_i(\sigma)$ and $\sum_{i=1}^\ell i b_i(\sigma) = \ell$, and, if $\sigma_1$ and $\sigma_2$ are conjugate reflections in $\Sigma$, then $b_i(\sigma_1)=b_i(\sigma_2)$, for all $i$, see the proof of \ref{lem:counting hyps} below.

Let $W$ denote the vector space spanned by $\cals$.  Then $\dim W^C = \dim W^g = |\cals/C| = \sum_{i=1}^\ell b_i(g)$.

\begin{lem}\label{lem:counting hyps}
    $\displaystyle
    |\calr_\calb(G)| 
    = \sum_{P\in\calb}\left(c\frac{\dim (W^{C_P})}{\dim(W)} -1\right).
    $
\end{lem}
\begin{proof}
  Suppose $H$ is a finite group acting transitively on a set $X$.  Let $A \subset H$ be a subset stable under
  conjugation, i.e., $A$ is a union of conjugacy classes of $H$.   Consider the set 
  $Y := \{(\sigma,x) \in  A \times X \mid \sigma \cdot x = x\}$.
  Then $|A_{x'}| = |A_{x''}|$ for all $x',x'' \in X$ where 
  $A_x = \{\sigma \in A \mid \sigma \cdot x = x\}$.  
  Thus $|Y| = \sum_{x \in X} |A_x| = |X| |A_{x}|$.   Also
  $|Y| = \sum_{\sigma \in A} |X^\sigma|$ where $X^\sigma = \{x \in X \mid \sigma\cdot x = x\}$.
  Thus $|A_{x}|= (\sum_{\sigma \in A} |X^\sigma|)/|X|$ for all $x \in X$.
  
    We apply this with $A = \calr_\calb(\Sigma)$, $X=\cals$ and with $x$ being the coset $G$.  
    Since $\sigma$ fixes $x$ if and only if $\sigma \in G$ we have
    $A_{x} = \calr_\calb(\Sigma) \cap G = \calr_\calb(G)$.  Hence, writing  $\FP(\sigma)=|\cals^\sigma|$ for the number of cosets left fixed by $\sigma$,
    we have 
\begin{align*}
    |\calr_\calb(G)|  &= \frac{1}{\ell}\sum_{\sigma \in \calr_\calb(\Sigma)}\FP(\sigma) =
           \frac{1}{\ell}\sum_{P\in\calb} \sum_{\substack{{g\in C_P}\\ {g \ne \id}}}\FP(g)\\
            &= \frac{1}{\ell}\sum_{P\in\calb} \big(\sum_{g \in C_P} \FP(g) - \FP(\id)\big)
            = \frac{1}{\ell}\sum_{P\in\calb}  \big(\sum_{g \in C_P} \FP(g) - \ell \big)\\
            &= \frac{1}{\ell}\sum_{P\in\calb} \big(c \dim(W^{C_P}) - \ell\big)
            = \sum_{P\in\calb} \big(c\frac{\dim(W^{C_P)}}{\dim W} - 1\big)
\end{align*}
as required.
\end{proof}

We let $\calr(\Sigma)$ denote the set of all reflections in $\Sigma$ and put $\calr(G) := \calr(\Sigma) \cap G$.
Summing the equation from Lemma~\ref{lem:counting hyps} over all the $\Sigma$ orbits on $\calp$  yields the following.

\begin{prop}\label{prop: counting reflections}
$$
    |\calr(G)| = \sum_{P\in\calp}\left(|C_P|\frac{\dim (W^{C_P})}{\dim(W)} -1\right).
 $$
\end{prop}

\section{The $G$-discriminant}\label{discriminant}
In this section we introduce the $G$-discriminant, $\Delta(G)$, 
a semi-invariant defined geometrically, and compute its degree 
which we show is equal to $\deg( \det M(\bth))$.

 Suppose that the hyperplane $P$ lies in the $\Sigma$ orbit $\calb$.
Define the integer
  $$
        e(G,P) = e(G,\calb) := \sum_{i\geq 2}\frac{ |C_P|}{2} b_i(g_P) (i-1).
    $$
    
    Note that $\sum_{i\geq 2}\frac{ |C_P|}{2} b_i(g_P) (i-1) = \sum_{i\geq 2}\frac{ |C_P|}{i} b_i(g_P) \binom{i}{2}$
    is an integer since $b_i(g_P) \ne 0$ implies that $i$ divides $|C_P|$.

\begin{defn}
 We define the $G$-discriminant, denoted $\Delta(G)$, by
 $$\Delta(G) = \prod_{P \in \calp} L_P^{e(G,P)}\ .$$ 
 Note that $\Delta(G)$ is only well-defined up to a scalar in the field $K$.
Recall that, if $\OK$ is a PID, then we 
assumed each linear form $L_P$ is primitive, and this implies that
$\Delta(G)\in \OK[V]$ is primitive, 
and thus $\Delta(G) $ is defined up to multiplication by a unit in $\OK$
in this case.
 \end{defn}
For each $\Sigma$ orbit $\calb$ on $\calp$ we define 
$\lambda_\calb := \prod_{P\in\calb}L_P$.
  Since $\Sigma$ permutes the planes in the orbit $\calb$ it follows that
$\lambda_{\calb}$ is a semi-invariant for $\Sigma$, i.e., there exists
some character $\chi$ of $\Sigma$ such that 
$\sigma \lambda_\calb = \chi(\sigma) \lambda_\calb$ for all $\sigma \in \Sigma$.  
Thus $\Delta(G)$, as a product of semi-invariants is itself a 
 semi-invariant for $\Sigma$.

\begin{remark}\label{G-discriminant for permutations}
  Suppose that $G \leq \Sn$.  The reflections in $\Sn$ are the transpositions
  and all transpositions are conjugate.  Thus for permutation groups 
  there is only one $\Sn$ orbit on $\calp$.  Recall the usual discriminant associated to the alternating group, $\disc=\prod_{1 \leq i < j \leq n}(x_i-x_j)$.  Then
  $\Delta(G) = \disc^e$ where
  $e$ is the number of orbits of size 2 under the action of one (any) 
  transposition on the set of left cosets $\Sn/G$.  Also, we note that for the alternating group $A_n$ in its usual representation, we have $e(A_n)=1$ so that $\Delta(A_n)=\disc$.  That is, our description of $\Delta(G)$ is a generalization of the usual discriminant for $A_n$.
\end{remark}

\begin{example}\label{ex:G-disc}
Consider the Shephard-Todd group $\Sigma = G(6,1,3)$   
consisting of the $3\times 3$ matrices  having one non-zero entry in each row and column with this entry being some $6^\text{th}$ root of unity. 
Let $\zeta$ denote a primitive $6^\text{th}$ root of unity and write
$K = \QQ(\zeta)$.
The subgroup $G$ of $\Sigma$ generated by the 2 matrices
$$ \begin{pmatrix}\zeta&0&0\\ 0&0&-1\\ 0&1&0\end{pmatrix},
\begin{pmatrix}\zeta&0&0\\ 0&1&0\\ 0&0&\zeta^{-1}\end{pmatrix}
$$ has order 216 and index 6.

There are $21$ reflecting hyperplanes in $G(6,1,3)$:
3 coordinate hyperplanes cut out by $x_i=0$ for $i=1,2,3$
and 18 cut out by the forms $x_i - \omega x_j$ where $1 \leq i < j \leq 3$
and $\omega=\zeta^k$ is some $6^\text{th}$ root of unity. 

The generators of the corresponding cyclic reflection subgroups form 2 conjugacy classes.
The group of reflections associated to $x_i=0$ is order 6 and consists
of diagonal matrices.  Let $\tau$ denote a generator for one of the reflection subgroups for first conjugacy class.  
A computation shows that 
$\tau$ acts on the cosets of $G$ via three 2-cycles, so $b_2(\tau)=3$ and
$b_i(\tau)=0$ for $i \neq 2$.
Therefore the exponent associated to these hyperplanes is 
$e(G,P)= \frac{6}{2}\cdot 3 (2-1) = 9$.
Therefore $x_1^9 x_2^9 x_3^9$ divides $\Delta(G)$.

The reflection group fixing the hyperplane cut out by $x_i - \omega x_j$ has order 2.  This group fixes 2 of the cosets and acts on the remaining 4 
cosets via two 2-cycles yielding an exponent of 2.
  Hence 
$\prod_{1\leq i < j\leq 3} \prod_{k=1}^6 (x_i - \zeta^k x_j)^2$ divides $\Delta(G)$.
Therefore $\Delta(G) = x_1^9 x_2^9 x_3^9 \prod_{1\leq i < j\leq 3} \prod_{k=1}^6 (x_i - \zeta^k x_j)^2$.
\end{example}

 Next we consider the degree of $\Delta(G)$.
\begin{thm}
We have
    $$
        \deg(\Delta(G)) = \frac{\ell}{2}(|\calr(\Sigma)| - |\calr(G)|)\ .
    $$
\end{thm}
\begin{proof}

We start with
    \begin{align*}
        e(G,P) &=  \frac{|C_P|}{2}\sum_{i\geq 1} b_i(i-1) 
                    = \frac{|C_P|}{2} (\sum_{i \geq 1} i b_i - \sum_{i \geq 1} b_i) \\
            &= \frac{|C_P|}{2} (\ell - \no(g_P))
            = \frac{|C_P| \ell }{2} - \frac{|C_P| \dim(W^{C_P})}{2}\ .
    \end{align*}

Hence
  \begin{align*}
      \sum_{P\in\calp}e(G,P) &= \frac{\ell}{2}\sum_{P\in\calp}|C_P| - \frac{1}{2}(\sum_{P\in\calp}|C_P|\dim(W^{C_P})) \\
        &= \frac{\ell}{2}(|\calr(\Sigma)| + \sum_{P\in\calp} 1) -\frac{1}{2}(\sum_{P\in\calp}|C_P|\dim(W^{C_P})) \\
        &= \frac{\ell}{2}(|\calr(\Sigma)|) - \frac{1}{2}\sum_{P\in\calp}|C_P|\dim(W^{C_P}) + \frac{\ell}{2} \sum_{P\in\calp} 1 \\
        & = \frac{\ell}{2}(|\calr(\Sigma)|) - \frac{\ell}{2}(\sum_{P\in\calp}|C_P|\frac{\dim(W^{C_P})}{\dim(W)} - 1) \\
        &= \frac{\ell}{2}(|\calr(\Sigma)|) - |\calr(G)|)
    \end{align*}
where the last displayed equation follows from Proposition~\ref{prop: counting reflections}.
\end{proof}

By the Theorem of Hochster-Eagon, $\Qi[V]^G$ is Cohen-Macaulay.  
Recall that for $1 \leq i \leq \ell$ the integer $a_i$ denotes the degree of the secondary invariant $\theta_i$.
We have the following corollary.
\begin{cor}
    $$
        \sum_{i=1}^{\ell} a_i = \sum_{P \in \calp} e(G,P)\ .
    $$
    In particular, $\deg (\det M(\bth)) = \deg (\Delta(G))$.
\end{cor}
\begin{proof}
This follows immediately from Stanley \cite[Corollary 4.3 and Corollary 4.4]{bible} which together imply that
    $$
        \ell |\calr(G)| +  2\sum_{i=1}^{\ell} a_i = \ell |\calr(\Sigma)|.
    $$
\end{proof}

\section{$\Delta(G)$ divides $\det M({\bth})$}\label{section:divides}

We have seen that $\det M(\bth)$ and $\Delta(G)$ share the same degree for any choice of $\bth$.
In this section we will show that $\Delta(G)$ divides $\det M(\bth)$
and thus $\det M(\bth)$ is a scalar multiple of $\Delta(G)$.  
The first step is the following theorem.

\begin{thm}
$\lambda_\calb^{e(G,\calb)}$ divides $\det M(\bth)$ in $\KK[V]$.  
\end{thm}

 \begin{proof}
 Note that $\Sigma$ acts on the set of rows of $M(\theta_*)$ by permutations.
 Let $W$ denote the $\CC$-vector space spanned by the rows of $M(\theta_*)$.  
 This is a subspace of $\oplus_{i=1}^\ell \CC[V]_{a_i}$.  
 We examine $W$ as a permutation representation of the cyclic group $C_P$ of order $c$.
 Since $C_P$ is abelian there is a basis $\set{r_1,r_2,\dots,r_\ell}$ of $W$ consisting of eigenvectors for $g_P$,
 i.e., working over $\CC$, we may diagonalize the action of $g_P$  on $W$.   Then $g_P r_i = \mu_i(g_P) r_i$ for some character $\mu_i$ of $C_P$.  We fix a choice of primitive $c^{\rm{th}}$ root of unity $\zeta_c\in\CC$ by writing $\det(g_P)=\zeta_c^{-1}$. 
 Note that since $g_p \in \GL(n,\OK)$, we have that $\zeta_c \in \OK$.

 The change of basis matrix $U$ which diagonalizes the 
 action of $g_P$ on W is an
invertible $\ell \times \ell$ matrix with entries in $K$.
Writing $\widetilde M := U M(\bth)$, we see that 
$\det (\widetilde M) = \kappa \det(M(\bth))$ where 
$\kappa = \det(U) \in K$.

We extend the set $\{L_P\}$
to a basis $\{L_P,L_2,\dots,L_n\}$ of $V^*$ where $\{L_2,\dots,L_n\}$ is dual to a basis for the reflecting hyperplane $P$.
Then $g_P \cdot L_P = \zeta_c L_P$ and $g_P \cdot L_i = L_i$
for $2 \leq i \leq n$.  
In this basis, every monomial is an eigenvector for $C_P$
with eigenvalue given by $\zeta_c^k$ with 
$0 \leq k \leq c-1$ where the degree of the monomial in $L_P$
is congruent to $k$ mod $c$.

 Now the cycle structure of $g_P$ acting on $\cals$ determines the permutation action of $C_P$ on $W$.
 A cycle of length $a$ corresponds to a transitive $a$ dimensional permutation $C_P$-subrepresentation of $W$.
 
  Diagonalizing the action of $g_P$ on this $a$ dimensional subrepresentation reveals eigenvectors for $g_P$
 with eigenvalues: $\zeta_{c}^{c/a}, \zeta_{c}^{2c/a},\dots,\zeta_{c}^{ac/a}=1$.
Thus the eigenvalue associated to $r_i$ is
$\mu_i(g_P)=\zeta_c^{cj/a}$ for some $j$ with
$0 \leq j < a$.  
In particular every monomial (when written in the basis $L_P, L_2,\dots, L_n$) appearing in 
each of the $\ell$ entries in $r_i$ must be divisible by $L_P^{c j/a}$.

Therefore a transitive permutation subrepresentation of $W$ of dimension $a$ 
contributes the factor 
$L_P^{\sum_{i=0}^{a-1} c i/a} = L_P^{c\binom{a}{2}/a}=L_P^{c(a-1)/2}$ to $\det \widetilde{M}$. 
 Taking all the cycles into account we see that
 $L_P^{(c/2)\sum_{i \geq 2} b_i(i-1)}$ divides 
 $\det \widetilde{M}$.
 Since for any two hyperplanes $P_1$ and $P_2$, the linear forms
 $L_{P_1}$ and $L_{P_2}$ are co-prime,  
  $\lambda_\calb^{e(G,\calb)}$ divides $\det \widetilde{M}$ in $\KK[V]$ for all orbits $\calb$.
  Therefore $\lambda_\calb^{e(G,\calb)}$ divides 
  $\det M(\bth)$ in $\KK[V]$ for all orbits $\calb$.
\end{proof}
  
Therefore
\begin{cor}\label{cor: main equation}
  $\Delta(G) = \prod_{\calb \subseteq \calp} \lambda_\calb^{e(G,\calb)}$ 
  divides $\det M(\bth)$ in $\KK[V]$.  
\end{cor}
\begin{proof}
This is clear since the $\lambda_\calb$ are pairwise co-prime in the UFD $K[V]$.  
\end{proof}

\begin{remark}
  If $\OK$ is a PID and $\bth$ is a set of secondaries 
  then $\mho(\bth) = \frac{\det M(\bth)}{\Delta(G)}$
   (up to multiplication by a unit in $\OK$). Note that this formula can be used to define $\mho(\bth)$ when $\OK$ is a PID because $\Delta(G)$ is primitive, thus Corollary~\ref{cor: main equation} implies it is the primitive part of $\det(M(\bth))$.
\end{remark}

\begin{example}\label{ex:G(2,2,n)}
Consider again the group $G \leq \Sigma = G(6,1,3)$ of order 216 defined in Example~\ref{ex:G-disc}.
Primary invariants are provided by the generators of 
$$K[V]^\Sigma = 
K[x_1^6 + x_2^6 + x_3^6,
x_1^{12} + x_2^{12} + x_3^{12},
x_1^{18} + x_2^{18} + x_3^{18}]\ .$$

For these 3 primaries, one choice of secondary invariants is provided by
$\theta_1 = 1$,
$\theta_2 = x_2^6 + x_3^6$,
$\theta_3 = x_1^3 x_2^3 x_3^3$,
$\theta_4 = x_2 ^{12} + 2 x_2^6 x_3^6 + x_3 ^{12}$,
$\theta_5 = x_1^3 x_2^9 x_3^3 + x_1^3 x_2^3 x_3^9$,
$\theta_6 = x_1^3 x_2 ^{15} x_3^3 + 2 x_1^3 x_2^9 x_3^9 
+ x_1^3 x_2^3 x_3^{15}$.

Choosing  coset representatives and forming
the matrix $M(\theta_*)$ we compute
$\det M(\theta_*) = 8 \Delta(G)$.
\end{example}

    \begin{remark}\label{evaluation}
  Note that if $\OK$ is a PID then for any choice of secondaries $\bth$ in $\OK[V]^G$, we may compute $\mho(\bth)$ efficiently by evaluating at a point $z \in V$:
  $$\mho(\bth) = \frac{\det M(\bth)\rvert_z}{\Delta(G)\rvert_z}\quad\text{(up to associates)}$$
for any point $z$ which does not lie on any reflecting hyperplane.  This allows us to compute
in the ring $\OK$ or $K$ rather than in $\OK[V]$ or $K[V]$.  
This makes the computation much simpler and faster.
\end{remark}

\begin{remark}
Recall $\disc=\prod_{i>j} (x_i-x_j)$.   
Note that when $\Sigma=\Sn$, 
Remark~\ref{G-discriminant for permutations}, implies
$\det M(\bth) = \mho(\bth) \disc^e$ where $e \in \NN$ and 
$\det M(\bth) = \disc^e$ if $G$ is generated by 
bi-reflections and $\bth$ is a universal set of 
secondaries.  
This provides a family of matrix 
factorizations for powers of the discriminant.
\end{remark}

\begin{example}\label{ex: mV2}
Consider the group $G=C_2=\{e,\sigma\}$ of order 2 and its $2a$ dimensional rational permutation representation $V=a\, V_2$ given by the direct sum of $a$ copies of its regular representation.  We fix coordinates by letting $\{x_1,y_1,\dots,x_a,y_a\}$ be a basis for $V^*$ where $\sigma$ acts by interchanging $x_i$ with $y_i$ for $i=1,\dots,a$.  Then ${C_2}$ is a subgroup of the  group $\Sigma$ of order $2^a$ generated by $\sigma_1,\dots,\sigma_a$ where $\sigma_i$ is the permutation that exchanges $x_i$ and $y_i$ and fixes $x_j$ and $y_j$ for $j\ne i$, $1 \le i,j \le a$.  Thus $\Sigma$ is a Young subgroup of $S_{2a}$ and $\sigma = \sigma_1 \cdots \sigma_a \in {C_2}$.  

For $a \geq 3$ the group ${C_2}$ contains no bireflections and so $\ZZ[V]^{C_2}$ is not Cohen-Macaulay.  Since the only prime dividing $|{C_2}|$ is 2 we must have that 2 is a bad prime and all other primes are good.

It is easy to see that $\QQ[V]^\Sigma = \QQ[r_1,\dots,r_a,s_1,\dots,s_a]$ where  $r_i = x_i + y_i$ and $s_i = x_i y_i$ for $1 \le i \le a$.  This follows from Stanley 
\cite[Corollary 4.4]{bible} since the $r_i$ and $s_i$ form a homogeneous system of parameters and the product of their degrees is the order of $\Sigma$.

Let $\calb$ denote the following set of ${C_2}$-orbit sums:
$$\calb := \{\calz(m) \mid \deg(m) \text{ is even}, m \text{ a monomial dividing }y_1 y_2 \cdots y_a\}\ .$$
Molien's Theorem tells that that the Hilbert polynomial of $\QQ[V]^{C_2}$ as a free module over $\QQ[V]^\Sigma$ is $\sum_{i=0}^{\lfloor a/2 \rfloor}\binom{a}{2i}\lambda^{2i}$.
It can be shown that the $2^{a-1}$ elements of $\calb$ 
form a set of secondaries for $\QQ[V]^{C_2}$.

We specialize to the case $a=4$.  Then $\mho(\calb)=8$,
as may be computed.  
Consider the $4$ invariants of degree $3$ of the form $q_{ijk}=\calz(y_iy_jy_k)=y_iy_jy_k+x_ix_jx_k$, $1\le i < j <k \le 4$.  
The ring $\QQ[4V_2]^{C_2}$ is Cohen-Macaulay, hence we have unique expressions
$$
q_{ijk} = \frac{1}{2}r_iq_{jk} + \frac{1}{2}r_jq_{ik} + \frac{1}{2}r_kq_{ij} - \frac{1}{2}r_ir_jr_k\, .
$$
Here $q_{ij}=x_ix_j+y_iy_j$. This implies that $\mathbb{F}_2[4V_2]^{C_2}$ is not Cohen-Macaulay. 
Computations show that 4 cubic invariants, such as all 4 of the $q_{ijk}$, are needed as secondaries for $\ZZ[4V_2]^{C_2}$, see Example~\ref{mv2_comp}.   
 \end{example}

\section{The Algorithm}\label{section:constructing}

In Section~\ref{existence} we gave a proof of the existence of a universal set of secondaries.  
In this section we give an algorithm to construct a set of universal secondaries for permutation representations. 
Combined with Remark~\ref{evaluation}, we obtain an algorithm which exactly determines all primes for which $\FF_p[V]^G$ is Cohen-Macaulay. Our algorithm uses linear algebra calculations over the integers. In this section, we return to considering $G$ as a permutation group with $G \leq \Sigma \leq S_n$ for a Young subgroup $\Sigma$.

Recall that $\tau(d)$ denotes the number of secondary
invariants of degree $d$ for $\QQ[V]^G$. Any set of universal secondaries forms a set of secondaries for
$\QQ[V]^G$ and thus must consist of $\tau(d)$ homogeneous
elements of degree $d$ for all $d$.

Let $\omega\in \ZZ[V]^\Sigma$ denote a non-zero element of degree 1.
For example, we may take $\omega$ to be the first elementary 
symmetric function: $\omega = x_1 + x_2 + \dots + x_n$.
Recall the map $\res_X : \QQ[V]^G \to \QQ^\ell$ from 
Section~\ref{good primes} given by restriction to a set $X=\phi^{-1}(\pi_\Sigma(z))$
where we require the stabilizer in $\Sigma$ of $z$ to be trivial.  
It follows from Remark~\ref{rem: lin ind over Q} that this map is a surjection.
For the algorithm, we need $z$ to lie in the lattice $\ZZ^n \subset \QQ^n$ and further 
to satisfy $\omega(z)\neq 0$.

We denote by $\QL_d$ the $\QQ$-vector space $\res_X(\QQ[V]^G_d)$.
Furthermore $\RL_d$ denotes the $\RR$-vector space $\res_X(\RR[V]^G_d)
= \QL_d \otimes_\ZZ \RR$.
Note that if $f \in \QQ[V]_{d-1}$ then $\omega f \in \QQ[V]_{d}$
with $\res_X(f)=\omega(z)^{-1} \res_X(\omega f) \in \QL_{d}$.  
Therefore, we have $\QL_{d-1}\subseteq \QL_{d}$ and a filtration
$$\{0\} \subsetneq \QL_0 \subseteq \QL_1 \subseteq \dots \subseteq \QL_t
=\QQ^\ell$$ where we use $t$ to denote $t=a_\ell=\deg(\theta_\ell)$.

Define $L_d$ to be the lattice of points in $L_d$ with integer coordinates, 
i.e., $L_d := \QL_d \cap \ZZ^\ell$.
Then 
$$\{0\} \subsetneq L_0 \subseteq L_1 \subseteq \dots \subseteq L_{t}
=\ZZ^\ell\ .$$

By construction, the quotients $L_d/L_{d-1}$ are torsion-free, so that each $L_d/L_{d-1}$ inherits the structure of a Euclidean lattice.  Geometrically we can view $L_d/L_{d-1}$ via the usual inner product on $\RR^\ell$, that is, there is a projection 
$\varepsilon_d: \QL_d \twoheadrightarrow \QL_{d}/\QL_{d-1}$
of $L_d \subseteq \RR^\ell$ onto the orthogonal complement of 
$\RL_{d-1}\subseteq \RR L_d$.  
Suppose that $A$ is a full rank lattice in $L_{d-1}$ with the covolume of $A$ in $\RL_{d-1}$ 
being $a$.  Let $B$ be a lattice in $L_d$ such that $A \cap B = \{0\}$ and 
$\rank(B) = \dim L_d - \dim L_{d-1}$.  Then the covolume of $A\oplus B$ in $\RL_d$ is 
$ab$ where $b$ is the covolume of $\varepsilon_d(B)$ in the orthogonal complement of $\RL_{d-1}$.
We will use this fact in the proof of Lemma~\ref{lem: covol prod}.

We also consider the lattice 
$M_d := \res_X(\ZZ[V]_d^G) \subseteq L_d$, the integer span of the evaluations of the orbit sums of monomials of degree $d$.  
Then $\varepsilon_d(M_d)$ is a lattice of full rank in $\QL_d/\QL_{d-1}$ and thus of finite index in $L_d/L_{d-1}$.

Finally, we introduce a sequence of lattices associated to a (any) set of secondaries 
$\bth \subset \ZZ[V]^G$.  We write $N_d^\theta$ to denote the sublattice of $M_d$
generated by the vectors $\{\res_X(\theta_j) \mid \deg(\theta_j) = d\}$.

Recall that the set $\{\res_X(\theta_1), \res_X(\theta_2),\dots,\res_X(\theta_\ell)\}$
must be linearly independent over $\QQ$.
Thus $\varepsilon_d(N^\theta_d)$ has full rank in $L_d/L_{d-1}$.

The dimension of the vector space
$\QL_d/\QL_{d-1}$ and the ranks of the lattices $L_d/L_{d-1}$, $M_d/(M_d \cap L_{d-1})$ and $N^\theta_d$
all coincide with the number $\tau(d)$. 
This is also the $\QQ$-dimension of the vector space $(\QQ[V]^G/I(\QQ))_d$
where $I(\QQ)$ is the ideal introduced in Section~\ref{sec: intro}.

Since $I(\QQ)_d \supseteq \omega\, \QQ[V]_{d-1}^G$ and since $\res_X(\omega f) = \omega(z)\res_X(f)$, 
it follows that 
$\res_X(I(\QQ)_d) = \QL_{d-1}$.
Thus we have the following commutative diagram of $\QQ$-vector spaces with exact rows and columns:

\begin{tikzcd}
0 \arrow[r] &I(\QQ)_d \arrow[r] \arrow[d, "\res_X"] &\QQ[V]_d^G \arrow[r,"\epsilon_d"] \arrow[d,"\res_X"]
     &(\QQ[V]^G/I(\QQ)_d \arrow[r] \arrow[d,dashed,"\kappa_d"] &0\\
0 \arrow[r] &\QL_{d-1} \arrow[r] \arrow[d] &\QL_d \arrow[r,"\varepsilon_d"] \arrow[d] &\QL_d/\QL_{d-1} \arrow[r] &0\\
& 0 & 0\\
\end{tikzcd}

Here $\varepsilon_d$ is the projection onto the orthogonal complement of $\QL_{d-1}$ in $\QL_d$.

  This diagram implies the existence of the surjective linear transformation
  $$\kappa_d : (\QQ[V]^G/I(\QQ))_d \longrightarrow \QL_d/\QL_{d-1}\ .$$
  In fact, since its domain and codomain share the same dimension, $\kappa_d$ is a vector space isomorphism.

  We now have suitable notation to describe the algorithm.

\begin{algorithm}
\caption{Find Universal Secondaries}\label{algorithm}
\begin{algorithmic}[1]
 \State $L_{-1} \leftarrow \{0\}$
\ForAll{$d \in \{0\dots t\}$}
\If {$\tau(d) = 0$}
\State $L_d \leftarrow L_{d-1}$
\Else
\State $\QL_d \leftarrow \res_X(\QQ[V]_d^G)$
\State $L_{d} \leftarrow \QL_d \cap \ZZ^\ell$
\State Compute a lattice basis $\{\bar{f}_1,\bar{f}_2,\dots,\bar{f}_{\tau(d)}\}$ 
for $\varepsilon_d(M_d) \subseteq L_d/L_{d-1}$.
\State Lift each $\bar{f}_i$ to an element
$f_i \in \ZZ[V]_d^G \cap (\kappa_d\circ\epsilon_d)^{-1}(\bar{f}_i)$
\State $\bth \leftarrow \bth \cup \{f_1,f_2,\dots,f_{\tau(d)}\}$.
\EndIf
\EndFor
\Return $\bth$.
\end{algorithmic}
\end{algorithm}

Step 8 of the algorithm requires finding a lattice basis for a lattice given a
generating set for the lattice.  
There are several algorithms for doing this, mainly variations of the LLL-algorithm. 
See Lenstra \cite[Section~14]{L} for a discussion of such algorithms.

The rest of this section is devoted to proving the correctness of Algorithm~1.

\begin{lem}\label{lem: division minimum}
    A set $\bth \subseteq \ZZ[V]^G$ of  secondaries is universal if
    $\det(M(\bth)\vert_z)$ divides $\det(M(\eta_*)\vert_z)$
    for every set of secondaries 
    $\eta_* \subseteq \ZZ[V]^G$. 
\end{lem}
\begin{proof}
Every bad prime must divide $\mho(\bth)$ for every set of secondaries $\bth$.
Conversely, $\bth$ is universal if no good prime divides $\mho(\bth)$.
Since we know from Section~\ref{existence} that there do exist sets of universal secondaries, it follows that
a set of secondaries $\bth \subseteq \ZZ[V]^G$ such that 
    $\mho(\bth)$ divides $\mho(\eta_*)$
    for every set of secondaries 
    $\eta_* \subseteq \ZZ[V]^G$ must be universal. 
Since the denominator of $(\det M(\bth)\vert_z)/(\Delta(G)\vert_z)$ is independent of $\bth$, the lemma follows.
\end{proof}

\begin{lem}\label{lem: lifting}
  Any homogeneous basis for the graded vector 
  space $\bigoplus_d\, {\QL_d}/{\QL_{d-1}}$ pulls back
  to a set of secondaries in $\QQ[V]^G$.
\end{lem}
\begin{proof}
  We have that $\dim_{\QQ}\,((\QQ[V]^G/I(\QQ))_d)=\tau(d)=\dim_{\QQ}(\QL_d/\QL_{d-1}).$
    Any $\QQ$-basis
of $\QL_d/\QL_{d-1}$ pulls back along the isomorphism $\kappa_d$
 to a $\QQ$-basis of $(\QQ[V]^G/I(\QQ)_d$. Amalgamating the pull backs along 
 $\epsilon_d$ for all those
$d$ such that $\QL_d/\QL_{d-1}$ is nontrivial, we obtain a graded $\QQ$-basis
for $\QQ[V]^G/I(\QQ)$, which by the graded Nakayama lemma lifts to a homogeneous
$\QQ[V]^\Sigma$-basis for $\QQ[V]^G$, i.e., to a set of secondaries.
\end{proof}

\begin{lem}\label{lem: covol prod}
    $$|\det(M(\bth))\vert_z| =\prod_{d \text{ such that }\tau(d)\ne 0} \cov \varepsilon_d(N^\theta_d)$$
\end{lem}
\begin{proof} 
 Consider the lattice $Y_d$ spanned by $\{\theta_i(z) \mid \deg(\theta_i)\leq d\}$. 
 Then $Y_d = \oplus_{i=1}^d N_i^\theta$ since $\{\theta_1(z),\theta_2(z),\dots,\theta_\ell(z)\}$ 
 is linearly independent over $\QQ$.
The columns of the matrix $M(\bth)\vert_z$  are the vectors $\res_X(\theta_j)$ for 
$j = 1,\dots,\ell$, and the number $|\det(M(\bth))\vert_z|$ is the covolume of
the lattice generated by these integer vectors.  Thus $|\det(M(\bth))\vert_z|=\cov Y_t$.
 
 We proceed by induction on $d$ to show that 
 $$\cov Y_d = \prod_{i \text{ such that }\tau(i)\ne 0\text{ and } i \leq d} 
 \cov \varepsilon_i(N_i^\theta)\ .$$
 For the base of the induction, $\epsilon_0:\RL_0 \to \RL_0$ is the identity map and 
 it is tautological that the rank one lattice $Y_0=N_0^\theta$ in $\RL_0$
 spanned by $\theta_1(z)$ has covolume $\cov \varepsilon_0(N_0^\theta)$.  
 For the induction step we suppose that the covolume of $Y_{d-1}$ is 
 $\prod_{i \text{ such that }\tau(i)\ne 0\text{ and } i \leq d-1} \cov \varepsilon_i(N^\theta_i)$
 and that $\tau(d) \neq 0$.  Then $Y_d = Y_{d-1} \oplus N^\theta_d$ has full rank in $\RL_d$,
 $Y_{d-1}$ has full rank in $\RL_{d-1}$
 and the covolume of $Y_d$ is given by the product of
 the covolume of $Y_{d-1}$ in $\RL_{d-1}$ with the covolume of $\varepsilon_d(N_d^\theta)$ in 
 the orthogonal complement of $\RL_{d-1}$ in $\RL_d$.  This completes the proof of the induction step. 
\end{proof}

We now show how the above three lemmas imply the correctness of Algorithm~1.
\begin{theorem}
    Algorithm~1 produces a set of universal secondaries.
\end{theorem}
\begin{proof}
The algorithm finds a vector space basis for the graded vector space 
$\bigoplus_d\, {\QL_d}/{\QL_{d-1}}$.
By Lemma~\ref{lem: lifting}, this pulls back to a set of secondaries $\bth$ for $\QQ[V]^G$.
For an arbitrary set of secondaries $\eta_*$,  the lattice 
$\varepsilon_d(N_d^\eta)$ 
has finite index $r_d$ in the lattice
$\varepsilon_d(M_d)$.  Then by Lemma~\ref{lem: covol prod},
\begin{align*}
|\det(M(\eta_*))\vert_z| &=
\prod_{d \text{ such that }\tau(d)\ne 0} \cov \varepsilon_d(N^\eta_d)
=\prod_{d \text{ such that }\tau(d)\ne 0} r_d \cov \varepsilon_d(M_d)
\ .
\end{align*}
Similarly,
\begin{align*}
|\det(M(\theta_*))\vert_z| &=
\prod_{d \text{ such that }\tau(d)\ne 0} \cov \varepsilon_d(N^\theta_d)
=\prod_{d \text{ such that }\tau(d)\ne 0} \cov \varepsilon_d(M_d)\ .
\end{align*}
Thus $$|\det(M(\eta_*))\vert_z| = \bigg(\prod_{d \text{ such that }\tau(d)\ne 0} r_d\bigg) |\det(M(\theta_*)\vert_z)|$$ and we see that 
for every set of secondaries $\eta_*$, the value $|\det(M(\eta_*))\vert_z|$ is a positive integer multiple of 
$|\det(M(\theta_*))\vert_z|$.
Hence, by Lemma~\ref{lem: division minimum}, the secondaries $\bth$ produced by
the algorithm form a universal set of secondaries.  
\end{proof}

\begin{remark}
Thus
$$\mho(G)= \frac{\prod_{\tau(d) \neq 0}\cov(\varepsilon_d(M_d))}{\Delta(G)\vert_z}$$
and this integer divides $\mho(\bth)$ for every choice of secondaries 
$\bth \subset \ZZ[V]^G$.
\end{remark}

\begin{remark}
The quotient $\ZZ[V]_d^G / I(\ZZ)_d = (\ZZ[V]^G / I(\ZZ))_d$ is a finitely 
generated abelian group for each degree $d$.  
Here $I(\ZZ)_d = \big(\oplus_{i=1}^d\ZZ[V]^\Sigma_i\cdot\ZZ[V]^G_{d-i})\big) \subset I(\QQ)_d$.
Comparing with $\QQ[V]^G_d/I(\QQ)_d$ we see that
$\ZZ[V]_d^G / I(\ZZ)_d$ has free rank $\tau(d)$.
If $\ZZ[V]_d^G / I(\ZZ)_d$ is a free $\ZZ$-module
for all $d$ then $\ZZ[V]^G$ is Cohen-Macaulay. To see this 
suppose that $\ZZ[V]_d^G / I(\ZZ)_d$ is a free $\ZZ$-module for all $d$.
Choose $\theta_{i_1},\dots,\theta_{i_{\tau(d)}}$ as a $\ZZ$-basis for $\ZZ[V]_d^G / I(\ZZ)_d$.
Then $\oplus_{j=1}^{\tau(d)} \ZZ\cdot \theta_{i_j} \oplus I(\ZZ)_d = \ZZ[V]_d^G$ for all $d$
showing that $\ZZ[V]^G$ is Cohen-Macaulay.

Thus
if $\ZZ[V]^G$ is not Cohen-Macaulay there is at least one degree $d$  such that the quotient 
$\ZZ[V]_d^G / I(\ZZ)_d$ is not 
torsion free.  In that case, there exists $f \in \ZZ[V]_d^G$ and $q \in \NN$ such that
$f \notin I(\ZZ)_d$ but $q f \in I(\ZZ)_d$.
 Such an element $f$ is needed as an additional secondary invariant for
$\ZZ[V]^G$ and a corresponding invariant is needed as an additional generator
for $\FF_p[V]^G$ for each $p$ dividing $q$.
\end{remark}

\begin{example}\label{mv2_comp}
  We revisit Example~\ref{ex: mV2} and use the notation from that example.
  Here $I(\QQ)$ is the ideal of $\QQ[V]^{C_2}$ generated by $r_i$, $s_i$, $1\le i \le a$. For $a \geq 3$, each $q_{ijk} \in \ZZ[a\,V_2]_3^{C_2}$.
Also  $q_{ijk}= \frac{1}{2} r_iq_{jk} + \frac{1}{2} r_jq_{ik} + \frac{1}{2} r_kq_{ij} - \frac{1}{2} r_{i}r_jr_k$.  Thus 
$2 q_{ijk} \in I_3(\ZZ) \subset I_3$ and $q_{ijk}$ is a torsion element in the quotient 
$\ZZ[a\,V_2]_3^{C_2}/I(\ZZ)_3$.  From this it follows that a set of secondaries for $\FF_2[a\,V_2]^{C_2}$
must include $\binom{a}{3}$ invariants of degree 3.  
However all rational secondaries have even degree.

Consider the case $a=3$ and use the point $z=(1,0,1,0,1,0)$ and coset representatives
$\gamma_1 = e,\gamma_2=\sigma_1, 
\gamma_3 = \sigma_2, \gamma_4 = \sigma_3$.
Then 
$w_1=z=(1,0,\,1,0,\,1,0)$,
  $w_2=(0,1,\,1,0,\,1,0)$,
  $w_3=(1,0,\,0,1,\,1,0)$ and 
  $w_4=(1,0,\,1,0,\,0,1)$.

We claim $\res_X(q_{123})=(1,0,0,0) \in L_2 \setminus M_2$
but $2\res_X(q_{123}) \in M_2$. Since $\res_X(r_1^2)=\res_X(r_2^2)=\res_X(r_3^2)=\res_X(r_1)$,
the lattice $M_2$ is generated by the set 
$$\{\res_X(r_1^2),\res_X(q_{12}),\res_X(q_{13}),\res_X(q_{23})\}
  = \{(1,1,1,1),(1,0,0,1),(1,0,1,0),(1,1,0,0)\}\ .$$
Thus
$\res_X(r_1 q_{23} + r_2 q_{13} + r_3 q_{12} - r_1 r_2 r_3) = (2,0,0,0)
 =\res_X(q_{23} + q_{13} + q_{12} - r_1^2) \in M_2$ and hence 
 $(1,0,0,0)=\res_X(q_{123}) \in L_2$.   However $(1,0,0,0) \notin M_2$.
 Therefore the covolume of $M_2$ in $\RR L_2$ is at least 2.

Suppose $a=4$ and choose $z=(1,0,1,0,1,0,1,0)$.  Then a computation yields $\mho({C_2})=8$.  Each of the 
4 elements $q_{ijk}$ is required as an extra secondary.
However 
$$\res_X(q_{123}) + \res_X(q_{124})+\res_X(q_{134})-
  \res_X(q_{12}+q_{13}+q_{14}-1) = \res_X(q_{234})$$
 even though $q_{234}$ does not lie in the 
 $\ZZ[V]^\Sigma$-module generated by 
 $$\{1,q_{12},q_{13},q_{14},q_{23},q_{24},q_{34},q_{123},q_{124},q_{134}\}\ .$$
Thus the restriction mapping preserves enough information to detect bad primes, but it
does not preserve all the structure of the $M_d$.  
\end{example}

  \begin{remark}
  In all the examples described above there exists a choice
  of universal secondaries with each secondary invariant
  being a single orbit sum.   Although we expect that this is not always possible we do not know of any example 
  failing this condition.
 \end{remark}

\section{Conclusion}\label{conclusion}
Any permutation representation of a finite group $G$ is defined over every ring.
Over a field of characteristic zero the ring of invariants is always 
Cohen-Macaulay.  The invariants $\ZZ[V]^G$ are closely related to 
$\FF_p[V]^G$ where $\FF_p$ is the field of prime order $p$.   
However, the latter ring of invariants may fail to be Cohen-Macaulay when 
$p$ divides the order of $G$.   
We defined a square matrix $M(\bth)$ associated to 
a set of secondary invariants for $\QQ[V]^G$,
$\bth=(\theta_1,\theta_2,\dots,\theta_\ell) \in \ZZ[V]^G$.
The determinant of $M(\bth)$ is an integer polynomial and
we denote its content by $\mho(\bth)$.
The mod $p$ reductions of the elements of $\bth$  form a set of
secondary invariants for 
the ring $\FF_p[V]^G$
if and only if $p$ does not divide $\mho(\bth)$.

  We showed also that there always exists a set of universal secondary invariants
$\bth=(\theta_1,\theta_2,\dots,\theta_\ell)\in\ZZ[V]^G$ such that their reductions mod $p$ yield
$\FF_p[V]^G = \oplus_{j=1}^\ell \FF_p[V]^{S_n} \btheta_j$ for all primes $p$ for
which $\FF_p[V]^G$ is Cohen-Macaulay.  We gave an
algorithm to construct a set of universal secondaries, $\bth$, using linear algebra over $\ZZ$, with the property that $\mho(\bth)$ divides $\mho(\eta_*)$ for every set of secondaries $\eta_* \in \ZZ[V]^G$.  Defining the deficiency of $G$ to be $\mho(G)=\mho(\bth)$ for such
universal secondaries, we have $\FF_p[V]^G$ is Cohen-Macaulay if and  only if $p$ does not divide the deficiency of $G$.   

For a subgroup $G$ of a finite unitary reflection group $\Sigma$,
defined over the ring of integers $\OK$ of a number field $K$
we defined a $\Sigma$-semi-invariant $\Delta(G)$.  
This semi-invariant was shown to be a scalar multiple of $\det M(\bth)$ for each sequence of secondary invariants $\theta_1,\theta_2,\dots,\theta_\ell \subset K[V]^G$.
Under the assumption $\OK$ is a PID,
the primitive part of $\det M(\bth)$ is $\mho(\bth)$ and 
thus $\det M(\bth)=\mho(\bth) \Delta(G)$ (up to units in $\OK$).
Evaluating both sides of this
equation at a point provides an efficient method to find either
$\mho(\bth)$ or $\mho(G)$.

{\bf {Acknowledgements.}}
The second author was partially supported by the
National Sciences and Engineering Research Council of Canada
and by the Canadian Defence Academy Research Programme. 
We thank Mike Roth for many helpful discussions.
The computer algebra system Magma~\cite{magma} 
was extremely helpful in computing and studying examples.
Sara Stephens implemented the algorithm described here in both Magma and Sage. 
This article was inspired by reading the preprint of Nicolas Borie and 
Nicolas Thi\'ery \cite{BT}.  We thank Ben Blum-Smith and 
Nicolas Thi\'ery for helpful comments.  
We also acknowledge the outstanding and detailed contributions of an 
anonymous referee which led to substantial improvements on earlier drafts
of this article.

\end{document}